\documentclass[3p,a4paper]{article} 
\usepackage{amsfonts,amsmath,amssymb,amsthm,stmaryrd,hyperref,bbm}
\usepackage[francais,english]{babel}
\usepackage[T1]{fontenc}
\usepackage[ansinew]{inputenc}  % si les car sont codes windows
\usepackage{graphics}
\usepackage{graphicx}
\usepackage{tikz}
\usepackage{multicol}
\usetikzlibrary{patterns}
\usepackage{amssymb,subfigure,lineno,url}
\usepackage{float}
\usepackage[draft]{showlabels}
\newtheorem{theorem}{Theorem}%[section]
\newtheorem{definition}{Definition}%[section]
\newtheorem{lemma}{Lemma}%[section]
\newtheorem{proposition}{Proposition}%[section]
\newtheorem{remark}{ Remark}%[section]

\begin{document}
\title{Decay of mass for a semilinear heat equation on Heisenberg group}

\author{Ahmad Z. Fino \footnote{\noindent 
College of Engineering and Technology, American University of the Middle East, Kuwait; ahmad.fino@aum.edu.kw.
 }}

\date{}
\maketitle

\begin{abstract}
In this paper, we are concerned with the Cauchy problem for the reaction-diffusion equation with time-dependent absorption
$u_{t}-\Delta_{\mathbb{H}}u=- k(t)u^p$ posed on $\mathbb{H}^n$, driven by the Heisenberg Laplacian and supplemented with a nonnegative integrable initial data, where $p>1$, $n\geq 1$, and $k:(0,\infty)\to(0,\infty)$ is a locally integrable function. We study the large time behavior of non-negative solutions and show that the nonlinear term determines the large time asymptotic for
$p\leq 1+2/Q,$ while the classical/anomalous diffusion effects win if $p>1+{2}/{Q}$, where $Q=2n+2$ is the homogeneous dimension of $\mathbb{H}^n$.
\end{abstract}



%=================================================================================={Introduction}
\section{Introduction}
In this work, we investigate the asymptotic behavior of solutions to the Cauchy problem for a nonlinear reaction-diffusion equation with time-dependent absorption, posed on the Heisenberg group
\begin{equation}\label{1}
\begin{array}{ll}
u_t-\Delta_{\mathbb{H}}u=- k(t)u^p,&\qquad {\eta\in \mathbb{H}^n,\,\,\,t>0,}
 \\{}\\ u(0,\eta)=u_{0}(\eta)\geq 0, &\qquad \eta\in \mathbb{H}^n,
 \end{array}
\end{equation}
where $n\geq1$, $p>1$, $k:(0,\infty)\to(0,\infty)$,  $k\in L^1_{loc}(0,\infty)$, $u_0\in L^1(\mathbb{H}^n)\cap C_0(\mathbb{H}^n)$,  $\Delta_{\mathbb{H}}$ is the Heisenberg Laplacian, and  $C_0(\mathbb{H}^n)$ denotes the space of all continuous functions tending to zero at infinity. The global existence and positivity of mild/classical solutions are obtained using the standard procedure via the comparison principle, combined with the fixed point theorem and regularity arguments.

%Reaction-diffusion equations with time-dependent absorption and nonlocal diffusion operators have attracted extensive attention in recent decades due to their wide range of applications and rich mathematical structure. A key question is how the interplay between nonlinear absorption and the diffusive effect governs the asymptotic behavior of solutions. 

\subsection{Historical background}
The study of mass decay in nonlinear heat equations is essential for understanding how diffusion and nonlinear absorption interact to influence the long-time dynamics of solutions. In the classical linear case,
$$\partial_t u=\Delta u, \qquad  u(0)=u_0\quad \text{in }\,\, \mathbb{R}^N,$$
the total mass is conserved for all \( t > 0 \), i.e., $\displaystyle \int_{\mathbb{R}^N} u(x, t) \, dx = \int_{\mathbb{R}^N} u_0(x) \, dx$, provided \( u_0 \in L^1(\mathbb{R}^N) \). This is easily verified by the representation formula
\begin{equation}\label{intro:Gaussian}
u(t,x)=e^{t\Delta} u_0(x)=[G(t,\cdot)*u_0](x), \quad 
G(t,x)=(4\pi t)^{-\frac{N}{2}}e^{-\frac{|x|^2}{4t}},
\end{equation}
where $e^{t\Delta}$ is the heat semigroup on $\mathbb{R}^N$.  However, the introduction of nonlinear absorption alters this behavior significantly. In particular, for the semilinear equation
\begin{equation*}
\partial_t u = \Delta u - u^p,\qquad  u(0)=u_0\geq 0\quad \text{in }\,\, \mathbb{R}^N,
\end{equation*}
where $p > 1$, the total mass 
$$M_{\mathbb{R}}(t) = \int_{\mathbb{R}^N} u(x, t) \, dx,$$
 typically decreases over time, and may decay to zero as \( t \to \infty \), depending on the exponent \( p \) and the spatial dimension \( N \). The analysis aims to uncover how the presence of the semilinear absorption term $-u^p$ modifies the dynamics compared to the purely linear setting. In the work of Gmira and V\'eron~\cite{GmiraVeron1984}, it is observed that comparing two supersolutions, one governed by the linear diffusion $(t,x)\mapsto e^{t\Delta} u_0(x)$, and the other governed by the nonlinear absorption \( t \mapsto (p-1)^{-\frac{1}{p-1}} t^{-\frac{1}{p-1}} \), reveals a fundamental threshold at \( p = 1 + \frac{2}{N} \) for integrable initial data. This threshold aligns with the Fujita phenomenon, originally developed in~\cite{Fujita1966}, and further studied in a series of works such as Hayakawa~\cite{Hayakawa1973}, Sugitani~\cite{Sugitani1975}, Kobayashi-Sirao-Tanaka~\cite{KobayashiSiraoTanaka1977}, and the monograph by Quittner and Souplet~\cite{QS}. The large-time behavior of the mass $M_{\mathbb{R}}(t)$ is crucial in understanding the asymptotic dynamics of the solution \( u \). Specifically,
\begin{itemize}
    \item If \( 1 < p \leq 1 + \frac{2}{N} \), then the mass decays to zero, that is, \( \lim_{t \to \infty} M_{\mathbb{R}}(t) = 0 \).
    \item If \( p > 1 + \frac{2}{N} \), the mass does not vanish at $t\to \infty$
and the solution $u$ behaves like the linear heat flow:
    \begin{equation*}
    \lim_{t \to \infty} \sup_{|x| \leq c t^{1/2}} \left| t^{N/2} u(t,x) - M_\infty G(t,x) \right| = 0,
    \end{equation*}
   where $c$ is an arbitrary positive constant.
\end{itemize}
A related decay estimate was also obtained by Fino and Karch~\cite{FinoKarch2010}, who proved that
\begin{equation*}
\lim_{t \to \infty} t^{\frac{N}{2} (1 - \frac{1}{q})} \| u(t, \cdot) - M_\infty G(t, \cdot) \|_{L^q(\mathbb{R}^N)} = 0, \quad \text{for all } 1 \leq q < \infty.
\end{equation*}
They also generalised the above results to the fractional reaction-diffusion equation
 \begin{equation}\label{eq1}
\partial_t u = -(-\Delta)^{\alpha/2} u - u^p, \qquad  u(0)=u_0\geq 0\quad \text{in }\,\, \mathbb{R}^N,
\end{equation}
where $(-\Delta)^{\alpha/2}$ stands for the fractional Laplacian of order $0<\alpha\leq 2$. It has been shown that the critical exponent for the large time behavior of solutions of \eqref{eq1} is $p=1+\alpha/N$, that is, the mass remains strictly positive for \( p > 1 + \frac{\alpha}{N} \) and vanishes as \( t \to \infty \) for \( p \leq 1 + \frac{\alpha}{N} \). For results concerning the asymptotic behavior of solutions involving higher-order elliptic operators in divergence form, we refer the reader to~\cite{KiraneQaf}.

A complementary contribution is given by Jleli and Samet~\cite{Jleli}, who rigorously analyzed the behavior of the mass function for the following nonlinear fractional diffusion equation with time-dependent absorption
 \begin{equation}\label{eq}
\partial_t u+t^\sigma(-\Delta)^{\alpha/2}  u=-h(t)u^p,\qquad  u(0)=u_0\geq 0\quad \text{in }\,\, \mathbb{R}^N,
\end{equation}
where $\sigma\geq0$, $p>1$, $h:(0,\infty)\to(0,\infty)$,  $h\in L^1_{loc}(0,\infty)$, $u_0\in L^1(\mathbb{R}^N)\cap C_0(\mathbb{R}^N)$. Under the assumptions
\begin{equation*}
\int_0^1 h(s) \, ds < \infty \quad \text{and} \quad \int_1^\infty s^{-\frac{N}{\alpha}(p - 1)(\sigma+1)} h(s) \, ds < \infty,
\end{equation*}
they established that \( M_\infty > 0 \). In contrast, when \( 1 < p \leq 1 + \frac{\alpha}{N(\sigma + 1)} \) and \( \inf_{t \geq 0} h(t) > 0 \), they prove that the solution's mass decays to zero, i.e., \( M_\infty = 0 \). If furthermore  \( h \in L^\infty(0,\infty) \), their sharp characterization is summarized as follows:
    \begin{itemize}
        \item If \( p > 1 + \frac{\alpha}{N(\sigma+1)} \), then \( M_\infty \in (0, \infty) \),
        \item If \( 1 < p \leq 1 + \frac{\alpha}{N(\sigma+1)} \), then \( M_\infty = 0 \).
    \end{itemize}
  An additional contribution to the understanding of mass dynamics in nonlocal equations is provided by Robles and Morales~\cite{RuvalcabaVilla2017}, who investigated the nonautonomous fractional reaction-diffusion equation
\begin{equation*}
\partial_t u = k(t)(-\Delta)^{\alpha/2} u - h(t) \varphi(u),\qquad  u(0)=u_0\geq 0\quad \text{in }\,\, \mathbb{R}^N,
\end{equation*}
with \( 0 < \alpha < 2 \), where \( k, h: [0, \infty) \to [0, \infty) \) are continuous, and \( \varphi: \mathbb{R} \to [0, \infty) \) is a convex, differentiable function. Under appropriate assumptions on the parameters, the authors established that the total mass remains strictly positive for all finite times and converges to a positive constant as $t\rightarrow\infty$, provided that $h\in L^1(0,\infty)$. These results highlight the diminishing influence of the absorption term over time and its role in preserving the long-time mass of the solution.

Further advancements by Kirane, Fino, and Ayoub~\cite{KFA} focused on a time-weighted semilinear equation involving a mixed diffusion operator \( L = -\Delta + (-\Delta)^{\alpha/2} \) and time-dependent absorption
\begin{equation*}
\partial_t u + t^\beta L u = -h(t) u^p,\qquad  u(0)=u_0\geq 0\quad \text{in }\,\, \mathbb{R}^N,
\end{equation*}
where $\beta\geq0$, $p>1$, $h:(0,\infty)\to(0,\infty)$,  $h\in L^1_{loc}(0,\infty)$, $u_0\in L^1(\mathbb{R}^N)\cap C_0(\mathbb{R}^N)$. They established a generalized Fujita-type critical exponent \( p_c = 1 + \frac{\alpha}{N(\beta + 1)} \) and demonstrated that solutions exhibit mass decay for \( p \leq p_c \), while persistence occurs for \( p > p_c \).

Very recently, Fino and Sobajima~\cite{FinoSobajima} examined the decay of mass with respect to
an invariant measure for semilinear heat equations in exterior domains
\begin{equation*}
\partial_t u - \Delta u + u^p = 0,
\end{equation*}
posed in exterior domains \( \Omega \subset \mathbb{R}^N \)  with Dirichlet boundary condition on \( \partial \Omega \). A notable aspect of their analysis is the use of an invariant measure \( \varphi(x)\, dx \), where \( \varphi \) is a positive harmonic function satisfying Dirichlet boundary conditions on \( \partial \Omega \). They showed that the total mass with respect to this weighted measure vanishes as \( t \to \infty \) if and only if \( 1 < p \leq \min\{2, 1 + \frac{2}{N} \} \). In contrast, when \( p > \min\{2, 1 + \frac{2}{N} \} \), all solutions are asymptotically free. For \( N \geq 3 \), the large-time asymptotic profile is described by a Gaussian-modified function.

%%%%%%%%%%%%%%%%%%%%%%

 \subsection{Motivation and Objectives}
Motivated by these results, we investigate the  Cauchy problem \eqref{1} on the Heisenberg group \( \mathbb{H}^n \). Our aim is to establish decay and persistence results analogous to those found in the Euclidean, but adapted to the subelliptic geometry of \( \mathbb{H}^n \). We derive critical exponent conditions under which the total mass either vanishes or remains positive as \( t \to \infty \), using semigroup estimates, scaling analysis, and suitable test function techniques tailored to the Heisenberg setting. 

Our results can be summarized as follows. A unique non-negative solution of the Cauchy problem \eqref{1} exists globally in time.  Hence, we  study the
decay properties of the mass 
$$ M_{\mathbb{H}}(t)= \int_{\mathbb{H}^n}u(t,\eta)\,d\eta,$$
 of the solution $u$ to problem \eqref{1}. We prove that $ \lim\limits_{t\to\infty}M_{\mathbb{H}}(t)=M_\infty >0$ for $p>1+{2}/{Q}$ (Theorem \ref{decay}), while $M_{\mathbb{H}}(t)$ tends to zero as $t\to\infty$ if $1<p\leq 1+2/Q$ (Theorem \ref{convto0}), where $Q=2n+2$ is the homogeneous dimension of $\mathbb{H}^n$. The proof of Theorem \ref{decay} is based on the $L^p-L^q$ estimates of solutions as well as the comparison principle, while for Theorem \ref{convto0}, our proof approach is based on the method of nonlinear capacity estimates or the so-called rescaled test function method. The nonlinear capacity method was introduced to prove the non-existence of global solutions by  Baras and Pierre \cite{Baras-Pierre}, then used by Baras and Kersner in \cite{Baras}; later on, it was developed by Zhang in \cite{Zhang} and Mitidieri and Pohozaev in \cite{19}.  It was also used by Kirane et al.  in \cite{Kirane-Guedda,KLT},  and Fino et al in \cite{Fino1,Fino3,Fino4,Fino5}.

 %%%%%%%%%%%%%%%%%%%%%%%

\subsection{Main results}

\begin{definition}[Mild solution]${}$\\
Let $u_0\in C_0(\mathbb{H}^n)$, $k\in L^1_{loc}(\mathbb{R})$, $n\geq1$, $p>1$, and $T>0$. We say that $u\in C([0,T),C_0(\mathbb{H}^n))$
is a mild solution of problem \eqref{1} if $u$ satisfies the following integral equation
\begin{equation}\label{IE}
    u(t,\eta)=S_{\mathbb{H}}(t) u_0(\eta)-\int_{0}^tS_{\mathbb{H}}(t-s)k(s) |u|^{p-1}u(s,\eta)\,ds,\quad \hbox{for all}\,\,\,\eta\in \mathbb{H}^n,\,t\in[0,T).
\end{equation}
More general, for all $t_0\geq 0$, we have
\begin{equation}\label{IEG}
    u(t,\eta)=S_{\mathbb{H}}(t-t_0) u(t_0,\eta)-\int_{t_0}^tS_{\mathbb{H}}(t-s)k(s) |u|^{p-1}u(s,\eta)\,ds,\quad \hbox{for all}\,\,\,\eta\in \mathbb{H}^n,\,t\in[t_0,T).
\end{equation}
If $u$ is a mild solution of \eqref{1} in $[0,T)$ for all $T>0$, then $u$ is called global-in-time mild solution of \eqref{1}.
\end{definition}
\begin{theorem}[Local existence]\label{Local}${}$\\
Given $u_0\in C_0(\mathbb{H}^n)$, $k\in L^1_{loc}(\mathbb{R})$, $n\geq1$, and $p>1,$ there exist a maximal
time $T_{\max}>0$ and a unique mild solution $u\in
C([0,T_{\max}),C_0(\mathbb{H}^n))$ to the problem \eqref{1}. Furthermore, either $T_{\max}=\infty$ or else $T_{\max}<\infty$ and $\|u(t)\|_{L^\infty(\mathbb{H}^n)}\rightarrow\infty$ as $t\rightarrow T_{\max}$. Moreover, if $u_0\in L^r(\mathbb{H}^n),$ for $1\leq r<\infty,$ then $u\in
C([0,T_{\max}),L^r(\mathbb{H}^n))$.
\end{theorem}
The local existence of the mild solution follows directly from the standard fixed point theorem in the Banach space $C_0(\mathbb{H}^n)$. The following comparison principle can be derived in a similar manner to \cite[Lemma~3.5.9]{CH}.
\begin{lemma}[The comparison principle]\label{Comparison}${}$\\
Let $u$ and $v$, respectively, be mild solutions of problem \eqref{1} with initial data $u_0$ and $v_0$, respectively. If $0\leq u_0\leq v_0$, and the function $k(t)$ does not change sign on \( (0, T_{\max}) \), then $0\leq u(x,t)\leq v(x,t)$ for almost every $\eta\in \mathbb{H}^n$ and for all $t\in [0,T_{\max})$.
\end{lemma}
The comparison principle ensures the non-negativity of the solution, highlighting its physical meaningfulness, particularly through the mass of the solution $M_{\mathbb{H}}(t)$. As the nonlinearity is of the absorbing type, then by applying the comparison principle, we can control the growth of the solution and extend its existence for all time. This leads to the following global existence result.
\begin{theorem}\label{Global}$(\mbox{Global existence})$\\
Given $0\leq u_0\in  L^1(\mathbb{H}^n)\cap C_0(\mathbb{H}^n)$, $k:(0,\infty)\to(0,\infty)$,  $k\in L^1_{loc}(0,\infty)$, and $p>1$. Then, problem \eqref{1} has a unique global mild solution $u\in C([0,\infty),L^1(\mathbb{H}^n)\cap C_0(\mathbb{H}^n))$ which satisfies $0\leq u(t,\eta)\leq\|u_0\|_{L^\infty}$.
\end{theorem}
\begin{remark}
Assume that $k\in C((0,\infty),(0,\infty))$. Then, by applying a standard bootstrap argument, any mild solution $u$ of \eqref{1} defined on an interval $[0, T]$ possesses $C^1$-regularity in time and $C^2$-regularity in space. Consequently, the solution $u$ satisfies the equation $\partial_t u-\Delta_{\mathbb{H}}u=- k(t)u^p$ in the pointwise sense.
\end{remark}
We deal with problem \eqref{1} and we study the decay of the ``mass''
\begin{equation}\label{mass}
    M_{\mathbb{H}}(t)= \int_{\mathbb{H}^n}u(t,\eta)\,d\eta=\int_{\mathbb{H}^n}u_0(\eta)\,d\eta -\int_0^t\int_{\mathbb{H}^n}k(s)u^p(s,\eta)\,d\eta\,ds.
\end{equation}
In order to obtain equality \eqref{mass}, it
suffices to integrate \eqref{IE} with respect to $\eta$, using property $(v)$ in Proposition \ref{properties} below, and applying Fubini's theorem.\\
Since we limit ourselves to non-negative solutions, the function
$M_{\mathbb{H}}(t)$ defined in \eqref{mass} is non-negative and non-increasing. As a result, $M_\infty = \lim\limits_{t\rightarrow\infty}M_{\mathbb{H}}(t)$, raising the natural question of whether $M_\infty =0$ or $M_\infty >0$. In our first theorem, we show that diffusion phenomena play a crucial role in determining the large-time asymptotic behavior of solutions to problem \eqref{1}.
\begin{theorem}\label{decay}
Let $u$ be a non-negative nontrivial global mild solution of
\eqref{1}. If
\begin{equation}\label{conditionh}
 \int_1^\infty  t^{-\frac{Q}{2}(p-1)}k(t)dt<\infty,
\end{equation} 
then
\begin{equation}\label{2.4}
\lim_{t\to\infty}M_{\mathbb{H}}(t)=M_\infty >0.
\end{equation}
Moreover, for all $q\in[1,\infty)$,
\begin{equation}\label{2.3}
    t^{\frac{Q}{2}\left(1-\frac{1}{q}\right)}\left\|u(t)-M_\infty
h_t\right\|_q\longrightarrow
0\quad\hbox{as}\;\; t\to\infty.
\end{equation}
\end{theorem}
\begin{remark}\label{remark1} The condition \eqref{conditionh} can be replaced by
$$ \int_{t_0}^\infty  t^{-\frac{Q}{2}(p-1)}k(t)dt<\infty,\qquad\hbox{for any}\,\,t_0>0.$$
On the other hand, if $k\in L^\infty(0,\infty)$ and $p> 1+\frac{2}{Q}$, then the condition \eqref{conditionh} in Theorem \ref{decay} is fulfilled. Indeed, as $ p>1+\frac{2}{Q}\Rightarrow\frac{Q(p-1)}{2}>1$, we have
$$\int_1^\infty  t^{-\frac{Q}{2}(p-1)}k(t)dt\leq \|k\|_\infty\int_1^\infty  t^{-\frac{Q}{2}(p-1)}dt<\infty.$$
\end{remark}

In the remaining range of $p$, the mass $M_{\mathbb{H}}(t)$ converges to zero and
this phenomena can be interpreted as the domination of nonlinear
effects in the large time asymptotic of solutions to \eqref{1}.
\begin{theorem}\label{convto0}
Let $u=u(x,t)$ be a non-negative global mild solution of problem
$\eqref{1}$. If $\displaystyle \inf_{t\geq 0}k(t)>0$ and $1<p\leq 1+\frac{2}{Q}$, then
$$
\lim_{t\to\infty}M_{\mathbb{H}}(t)=M_\infty=0.
$$
\end{theorem}
An immediate consequence of Theorems \ref{decay} and \ref{convto0}, using Remark \ref{remark1}, is the following

\begin{theorem}\label{general}
Let $u=u(x,t)$ be a non-negative nontrivial global mild solution of problem
$\eqref{1}$. If $k\in L^\infty(0,\infty)$ and $\displaystyle \inf_{t\geq 0}k(t)>0$, then we have two cases:
\begin{itemize}
\item If $p>1+\frac{2}{Q}$, then $M_\infty>0$. Moreover, \eqref{2.3} holds for all $q\in[1,\infty)$.
\item If $1<p\leq 1+\frac{2}{Q}$, then $M_\infty=0$.
\end{itemize}
\end{theorem}

The paper is organized as follows. In Section \ref{sec2}, we present some definitions, terminologies, and preliminary results concerning the Heisenberg group, the sub-Laplacian operator and its heat kernel. Sections \ref{sec3} and \ref{sec4} are  devoted to prove our main results.

\medskip
\noindent\textbf{Notation:}  We shall employ the following notational conventions in what follows. For $\Omega \subset \mathbb{R}$, the characteristic function on $\Omega$ is denoted by $\mathbbm{1}_\Omega$. For any nonnegative functions $g$ and $h$, we write $f\lesssim g$ whenever there exist positive constant $C>0$ such that $f\leqslant Cg$. We use $C$ as a generic constant, which may change value from line to line. Norms in Lebesgue spaces $L^p$ are written as $\| \cdot \|_p$, for $p \in [1, \infty]$.

%%%%%%%%%%%%%%%%%%%%%%%%%%%%%%%%%%%%%%%%%%%%%%%%%%%%%%%%%%%%%%%%%%%%%%%

\section{Preliminaries}\label{sec2}

\subsection{Heisenberg group}
The Heisenberg group $\mathbb{H}^n$  is the space $\mathbb{R}^{2n+1}=\mathbb{R}^n\times \mathbb{R}^n\times \mathbb{R}$ equipped with the group operation
$$\eta\circ\eta^\prime=(x+x^\prime,y+y^\prime,\tau+\tau^\prime+2(x\cdotp y^\prime-x^\prime\cdotp y)),$$
where $\eta=(x,y,\tau)$, $\eta^\prime=(x^\prime,y^\prime,\tau^\prime)$, and $\cdotp$ is the standard scalar product in $\mathbb{R}^n$. Let us denote the parabolic dilation in $\mathbb{R}^{2n+1}$ by $\delta_\lambda$, namely, $\delta_\lambda(\eta)=(\lambda x, \lambda y, \lambda^2 \tau)$ for any $\lambda>0$, $\eta=(x,y,\tau)\in\mathbb{H}^n$. The Jacobian determinant of $\delta_\lambda$ is $\lambda^Q$, where $Q=2n+2$ is the homogeneous dimension of $\mathbb{H}^n$. A direct calculation shows that $\delta_\lambda$ is an automorphism of $\mathbb{H}^n$ for every $\lambda>0$, and therefore  $\mathbb{H}^n=(\mathbb{R}^{2n+1},\circ,\delta_\lambda)$ is a homogeneous Lie group on $\mathbb{R}^{2n+1}$.

The homogeneous Heisenberg norm (also called Kor\'anyi norm) is derived from an anisotropic dilation on the Heisenberg group and defined by
$$|\eta|_{_{\mathbb{H}}}=\left(\left(\sum_{i=1}^n (x_i^2+y_i^2)\right)^2+\tau^2\right)^{\frac{1}{4}}=\left((|x|^2+|y|^2)^2+\tau^2\right)^{\frac{1}{4}},$$
where $|\cdotp|$ is the Euclidean norm associated to $\mathbb{R}^n$. The associated Kor\'anyi distance between two points $\eta$ and $\xi$ of $\mathbb{H}$ is defined by
$$d_{_{\mathbb{H}}}(\eta,\xi)=|\xi^{-1}\circ\eta|_{_{\mathbb{H}}},\quad \eta,\xi\in\mathbb{H},$$
where $\xi^{-1}$ denotes the inverse of $\xi$ with respect to the group action, i.e. $\xi^{-1}=-\xi$. This metric induces a topology on $\mathbb{H}^n$. Thus, we can define the Heisenberg ball of $\mathbb{H}^n$, centered
at $\eta$ and with radius $r>0$, as $B_{\mathbb{H}}(\eta,r)=\{\xi\in\mathbb{H}^n:\,d_{_{\mathbb{H}}}(\eta,\xi)<r\}$.

The Heisenberg convolution between any two regular functions $f$ and $g$ is defined by
$$(f\ast_{_{\mathbb{H}}}g)(\eta)=\int_{\mathbb{H}^n}f(\eta\circ\xi^{-1})g(\xi)\,d\xi=\int_{\mathbb{H}^n}f(\xi)g(\xi^{-1}\circ\eta)\,d\xi.$$

The left-invariant vector fields that span the Lie algebra are given by
$$X_i=\partial_{x_i}-2 y_i\partial_\tau,\qquad Y_i=\partial_{y_i}+2 x_i\partial_\tau.$$
The Heisenberg gradient is given by
\begin{equation}\label{48}
\nabla_{\mathbb{H}}=(X_1,\dots,X_n,Y_1,\dots,Y_n),
\end{equation}
and the sub-Laplacian (also referred to as Kohn Laplacian) is defined as
\begin{equation}\label{40}
\Delta_{\mathbb{H}}=\sum_{i=1}^n(X_i^2+Y_i^2)=\Delta_x+\Delta_y+4(|x|^2+|y|^2)\partial_\tau^2+4\sum_{i=1}^{n}\left(x_i\partial_{y_i\tau}^2-y_i\partial_{x_i\tau}^2\right),
\end{equation}
where $\Delta_x=\nabla_x\cdotp\nabla_x$ and $\Delta_y=\nabla_y\cdotp\nabla_y$ stand for the Laplace operators on $\mathbb{R}^n$.

\subsection{Heat kernel}
We recall the definition and some properties related to the heat kernel associated to $-\Delta_{\mathbb{H}}$ on the Heisenberg group.
\begin{proposition}\label{Heat}\cite[Theorem~8]{BahouriGallagher}${}$\\
There exists a function $h\in\mathcal{S}(\mathbb{H}^n)$ such that if u denotes the
solution of the free heat equation on the Heisenberg group
\begin{equation}\label{}
\begin{array}{ll}
\partial_t u-\Delta_{\mathbb{H}}u=0,&\qquad {\eta\in \mathbb{H}^n,\,\,\,t>0,}
 \\{}\\ u(\eta,0)=u_{0}(\eta), &\qquad \eta\in \mathbb{H}^n,
 \end{array}
\end{equation}
then we have
$$u(\cdotp,t)=h_t\ast_{_{\mathbb{H}}} u_0,$$
where $\ast_{_{\mathbb{H}}}$ denotes the convolution on the Heisenberg group defined above, while the heat kernel $(t,\eta)\in(0,\infty)\times\mathbb{H}^n\mapsto h_t(\eta)$ associated to $-\Delta_{\mathbb{H}}$ is defined by
$$h_t(\eta)=\frac{1}{t^{n+1}}h\left(\frac{x}{\sqrt{t}},\frac{y}{\sqrt{t}},\frac{s}{t} \right),\qquad\hbox{for all}\,\,\,\eta=(x,y,s)\in\mathbb{H}^n, \,\,t>0.$$
\end{proposition} 
The following proposition follows directly from \cite[Theorem~3.1]{Folland}, \cite[Corollary~1.70]{Follandstein}, \cite[Proposition~1.68]{Follandstein}, and \cite[Corollary~2.3]{Pazy}.
\begin{proposition}\label{properties}${}$\\
There is a unique semigroup $(S_{\mathbb{H}}(t))_{t>0}$ generated by $\Delta_{\mathbb{H}}$ and satisfying the following properties:
\begin{itemize}
\item[$(i)$] $(S_{\mathbb{H}}(t))_{t>0}$ is a contraction semigroup on $L^p(\mathbb{H}^n)$, $1\leq p\leq \infty$, which is strongly continuous for $p<\infty$. 
\item[$(ii)$] $(S_{\mathbb{H}}(t))_{t>0}$  is a strongly continuous semigroup on $C_0(\mathbb{H}^n)$.
\item[$(iii)$] For every $v\in X$, where $X$ is either $L^p(\mathbb{H}^n)$ for $1\leq p< \infty$ or $C_0(\mathbb{H}^n)$, the map $t\longmapsto S(t)v$ is continuous from $[0,\infty)$ into $X$.
\item[$(iv)$] $S_{\mathbb{H}}(t)f=h_t\ast_{_{\mathbb{H}}} f$, for all $f\in L^p(\mathbb{H}^n)$, $1\leq p\leq \infty$, $t>0$.
\item[$(v)$] $\displaystyle\int_{\mathbb{H}^n}h_t(\eta)\,d\eta=1$, for all $t>0$.
\item[$(vi)$] $h_t(\eta)\geq0$, for all $\eta \in \mathbb{H}^n, t>0.$ 
\item[$(vii)$] $h_{r^2t}(rx,ry,r^2s)=r^{-Q}h_{t}(x,y,s)$, for all $r,t>0$, $(x, y,s) \in \mathbb{H}^n$. 
\item[$(viii)$] $h_t(\eta)=h_t(\eta^{-1})$, for all $\eta \in \mathbb{H}^n$, $t>0$.
\item[$(ix)$] $h_t\ast_{_{\mathbb{H}}} h_s=h_{t+s}$, for all $s,t>0$. 
\end{itemize}
\end{proposition} 
\begin{proposition}\cite[Theorem~1.3.2]{Randall}${}$\\
The function \( h_t \) is given by
\[
h_t(\eta) = \frac{1}{(2\pi)^{n+2} 2^n} \int_{\mathbb{R}} \left(\frac{\lambda}{\sinh(t\lambda)}\right)^n
\exp\left( -\frac{|z|^2 \lambda}{4 \tanh(t\lambda)} \right) 
e^{i\lambda \tau} \, d\lambda, \quad \hbox{for all}\,\,\eta=(z, \tau) \in \mathbb{R}^{2n} \times \mathbb{R}.
\]
\end{proposition} 
\begin{proposition}\label{pointwiseupperbound} \cite[Theorems~2,4]{Jerison} or \cite[Theorems~IV.4.2-4.3]{Varapoulos}${}$\\
Let $h_t$ be the heat kernel associated to $-\Delta_{\mathbb{H}}$. Then there exist two positive constants $c_\star, C_\star$ such that the kernel can be estimated as follows:
\[
c_\star t^{-\frac{Q}{2}} \exp\left( -\frac{C_\star |\eta|_H^2}{t} \right) 
\leq h_t(\eta) 
\leq C_\star t^{-\frac{Q}{2}} \exp\left( -\frac{c_\star |\eta|_H^2}{t} \right) 
\]
for any $t > 0$ and $\eta \in \mathbb{H}^n$.
\end{proposition} 
\begin{proposition}\label{pointwiseheatestimation}\cite[Theorem~1]{Jerison} or \cite[Theorem~IV.4.2]{Varapoulos}${}$\\
Let $h_t$ be the heat kernel associated to $-\Delta_{\mathbb{H}}$. Then there exist positive constants $c_1$ and $C_{I,l}$ depending $-\Delta_{\mathbb{H}}$ such that
$$|\partial_t^lX_I h_t(\eta)|\leq C_{I,l}t^{-l-\frac{|I|}{2}-\frac{Q}{2}}e^{-\frac{c_1|\eta|^2_{_{\mathbb{H}}}}{t}},\qquad\hbox{for all}\,\,\,\eta\in\mathbb{H}^n,\,\,t>0,$$
where $I=(i_1,\dots,i_m)$ with $|I|=m$ and $X_I=X_{i_1}X_{i_2}\dots X_{i_m}$.
\end{proposition} 
The following result is a direct consequence of Proposition \ref{pointwiseheatestimation}.
\begin{lemma}\label{estimateheatkernel}
Let $h_t$ be the heat kernel associated to $-\Delta_{\mathbb{H}}$. Then there exist positive constants $C_1$ and $C_2$ depending $-\Delta_{\mathbb{H}}$ such that
$$\left\|\nabla_{\mathbb{H}} h_t(\cdotp)\right\|_1\leq C t^{-\frac{1}{2}},\qquad \left\|T h_t(\cdotp)\right\|_1\leq C t^{-\frac{1}{2}},\qquad t>0,$$
where $\nabla_{\mathbb{H}}$ is defined by \eqref{48} and $T=\partial_\tau$.
\end{lemma}
\begin{lemma}\label{Younginequality}\cite[Proposition~1.8]{Follandstein}${}$\\
Suppose $1\leq p,q,r\leq \infty$ and $\frac{1}{p}+\frac{1}{q}=1+\frac{1}{r}$. If $f\in L^p(\mathbb{H}^n)$ and $g\in L^q(\mathbb{H}^n)$, then their convolution $f\ast_{_{\mathbb{H}}}g$ belongs to $L^r(\mathbb{H}^n)$, and the following inequality holds
$$\|f\ast_{_{\mathbb{H}}}g\|_r\leq \|f\|_p\|g\|_q,$$
which is known as Young's inequality.
\end{lemma}
Thus, by applying Young's inequality (Lemma \ref{Younginequality}) for convolutions and using Propositon \ref{pointwiseheatestimation} along with the homogeneity property 
$$|(\lambda x,\lambda y,\lambda^2 \tau)|_{_{\mathbb{H}}}=\lambda |\eta|_{_{\mathbb{H}}},\quad \hbox{for}\,\, \eta=(x,y,\tau)\in\mathbb{H}^n,\,\,\lambda>0,$$ 
we obtain the following estimate.
\begin{lemma}\label{Lp-Lqestimate}($L^p-L^q$ estimate)${}$\\
 Suppose $1\leq p\leq q\leq \infty$. Then there exists a positive constant $C$ such that for every $f\in L^p(\mathbb{H}^n)$, the following inequality holds:
$$\|S_{\mathbb{H}}(t)f\|_q\leq C\,t^{-\frac{Q}{2}(\frac{1}{p}-\frac{1}{q})}\|f\|_p,\qquad t>0.$$
\end{lemma}
To express the zero-order Taylor expansion with integral remainder on the Heisenberg group in a straightforward and comprehensible manner, we provide it here along with a detailed proof.
\begin{lemma}\label{Taylorexpansion}
Let $f:\mathbb{H}^n\rightarrow \mathbb{R}$ be a smooth function. For any point $\eta=(x_0,y_0,\tau_0)\in \mathbb{H}^n$ and a small perturbation $\xi=(x,y,\tau)\in \mathbb{H}^n$, the zero-order Taylor expansion of 
$f(\eta\circ\xi)$ around $\eta$ with an integral remainder is given by
$$f(\eta\circ\xi)=f(\eta)+\int_0^1(x,y,2s \tau)^{T}\cdot\widetilde{\nabla}_{\mathbb{H}}f(\eta\circ\delta_s(\xi))\,ds,$$
where $\widetilde{\nabla}_{\mathbb{H}}=(\nabla_{\mathbb{H}},\partial_\tau)^{T}$.
\end{lemma}
\begin{proof}
Using the fundamental theorem of calculus, we express the function $f$ along the curve $\gamma(s)=\eta\circ \delta_s(\xi)$ for $s\in[0,1]$,
$$f(\eta\circ\xi)-f(\eta)=\int_0^1\frac{d}{ds}f(\eta\circ\delta_s(\xi))\,ds.$$
Applying the chain rule in the Heisenberg group,
\begin{eqnarray*}
\frac{d}{ds}f(\eta\circ\delta_s(\xi))&=&\sum_{i=1}^nX_if(\eta\circ\delta_s(\xi))\frac{d}{ds}(sx_i)+\sum_{i=1}^nY_if(\eta\circ\delta_s(\xi))\frac{d}{ds}(sy_i)+\partial_\tau f(\eta\circ\delta_s(\xi))\frac{d}{ds}(s^2\tau)\\
&=&\sum_{i=1}^nx_i X_if(\eta\circ\delta_s(\xi))+\sum_{i=1}^ny_i Y_if(\eta\circ\delta_s(\xi))+2 s\tau \partial_\tau f(\eta\circ\delta_s(\xi))\\
&=&(x,y,2s \tau)^{T}\cdot\widetilde{\nabla}_{\mathbb{H}}f(\eta\circ\delta_s(\xi)).
\end{eqnarray*}
Integrating both sides from $0$ to $1$, the result follows.
\end{proof}
\begin{lemma}\label{Taylor}
Let $g\in L^1(\mathbb{H}^n)$ and put $\displaystyle M_g=\int_{\mathbb{H}^n}g(\eta)\,d\eta$. We have 
\begin{equation}\label{TaylorInequality1}
\lim\limits_{t\to\infty}\|h_t\ast_{_{\mathbb{H}}} g-M_gh_t\|_1=0.
\end{equation}
If, in addition, $(|\eta|_{_{\mathbb{H}}}+|\eta|^2_{_{\mathbb{H}}} ) g(\eta) \in L^1(\mathbb{H}^n)$, then
\begin{equation}\label{TaylorInequality2}
\|h_t\ast_{_{\mathbb{H}}} g-M_gh_t\|_1\leq C\,t^{-1/2}\|(|\eta|_{_{\mathbb{H}}}+|\eta|^2_{_{\mathbb{H}}} )  g(\eta)\|_1,\qquad\hbox{for all}\,\, \,t>0.
\end{equation}
\end{lemma}
\begin{proof}
We adapt the technique used in \cite[Proposition~48.6]{QS} and \cite[Lemma~6]{KFA}. We first establish \eqref{TaylorInequality2} by supposing $g\in L^1(\mathbb{H}^n,\,(1+|\eta|^2_{_{\mathbb{H}}} )\,d\eta)$. Using Taylor's expansion (Lemma \ref{Taylorexpansion}), Fubini's theorem and Lemma \ref{estimateheatkernel}, we have
\begin{eqnarray*}
\|h_t\ast_{_{\mathbb{H}}} g-M_gh_t\|_1&=&\left\|\int_{\mathbb{H}^n}\left(h_t(\eta\circ \xi^{-1})-h_t(\eta)\right)g(\xi)\,d\xi\right\|_1\\
&\leq&\int_0^1\int_{\mathbb{H}^n}\left\|\nabla_{\mathbb{H}} h_t(\eta\circ\delta_s(\xi))\right\|_1 |(x,y)||g(\xi)|\,d\xi\,ds+2\int_0^1\int_{\mathbb{H}^n}\left\|\partial_\tau h_t(\eta\circ\delta_s(\xi))\right\|_1 s|\tau||g(\xi)|\,d\xi\,ds\\
&\leq&C\,t^{-\frac{1}{2}}\int_0^1\int_{\mathbb{H}^n} (|(x,y)|+|\tau|)|g(\xi)|\,d\xi\,ds\\
&\leq&C\,t^{-\frac{1}{2}}\|(|\eta|_{_{\mathbb{H}}}+|\eta|^2_{_{\mathbb{H}}} ) g(\eta)\|_1.
\end{eqnarray*}
Let us next prove \eqref{TaylorInequality1}. For $g\in L^1(\mathbb{H}^n)$, using e.g. \cite[Theorem~2.2]{Baldi}, there exists a sequence $\{g_j\}\in \mathcal{D}(\mathbb{H}^n)$ such that $g_j\rightarrow g$ in $L^1(\mathbb{H}^n)$. For each $j$, using the fact that $\|h_t\|_1=1$, and applying \eqref{TaylorInequality2} to $g_j$, we have
\begin{eqnarray*}
\|h_t\ast_{_{\mathbb{H}}} g-M_gh_t\|_1&\leq&\|h_t\ast_{_{\mathbb{H}}} g-h_t\ast_{_{\mathbb{H}}} g_j\|_1+\|h_t\ast_{_{\mathbb{H}}} g_j-M_{g_j}h_t\|_1+\|M_{g_j}h_t-M_gh_t\|_1\\
&\leq&\|g-g_j\|_1\|h_t\|_1+\|h_t\ast_{_{\mathbb{H}}}g_j-M_{g_j}h_t\|_1+|M_{g_j}-M_g|\|h_t\|_1\\
&\leq&2\|g-g_j\|_1+C\,t^{-\frac{1}{2}}\|(|\eta|_{_{\mathbb{H}}}+|\eta|^2_{_{\mathbb{H}}} ) g_j(\eta)\|_1.
\end{eqnarray*}
Hence,
$$\limsup_{t\rightarrow\infty}\|h_t\ast_{_{\mathbb{H}}} g-M_gh_t\|_1\leq 2\|g-g_j\|_1,$$
and the conclusion follows by passing to the limit as $j\rightarrow\infty$.
\end{proof} 
In preparation for later use, we introduce and describe  the properties of the cut-off functions $\{\varphi_R\}_{R>0}$
defined as 
\begin{equation}\label{varphi_R}
\varphi_R(t,\eta)= \Phi^{\ell}\big(\xi_R(t,\eta)\big),\qquad \xi_R(t,\eta)=\frac{t+|\eta|_{_{\mathbb{H}}}^2}{R},\quad\eta=(x,y,\tau)\in\mathbb{H}^n,\,\,t\geq0,
\end{equation}
where $\ell=2p'$, $p'=p/(p-1)$ is the H\"older conjugate of $p$, and $\Phi\in \mathcal{C}^\infty(\mathbb{R})$ is a smooth non-increasing function 
satisfying $\mathbbm{1}_{[1,\infty)} \leq \Phi\leq \mathbbm{1}_{[\frac{1}{2},\infty)}$.  Furthermore, we define 
$$\varphi_R^*(t,\eta)=\Phi^\ell_*\big(\xi_R(t,\eta)\big)\quad\hbox{where}\quad \Phi_*=\mathbbm{1}_{[\frac{1}{2},1]}\Phi.$$
\begin{lemma}\label{lem:testfunction}
Define the family of cut-off functions $\{\varphi_R\}_{R>0}$
as in \eqref{varphi_R}.
Then the following inequality holds:
\[
|\partial_t\varphi_R(t,\eta)|
+
|\Delta_{\mathbb{H}}\varphi_R(t,\eta)|\leq \frac{C}{R}
\left(\varphi_R^*(t,\eta)\right)^{\frac{1}{p}}.
\]
\end{lemma}
\begin{proof}
It is easy to see that $|\partial_t\varphi_R|
\leq \frac{\ell\|\Phi_*'\|_{L^\infty}}{R}\varphi_R^*(t,\eta)^{\frac{1}{p}}$. Therefore we focus our attention to $\Delta_{\mathbb{H}}\varphi_R$. In fact,
\begin{eqnarray*}
\left|\Delta_{\mathbb{H}}\varphi_R(t,\eta)\right|&\leq&\left|\Delta_x\Phi^\ell\big(\xi_R(t,\eta)\big)\right|+\,\left|\Delta_y\Phi^{\ell}\big(\xi_R(t,\eta)\big)\right|+\,4(|x|^2+|y|^2)\left|\partial_\tau^2\Phi^{\ell}\big(\xi_R(t,\eta)\big)\right|\\
&{}&+\,4\sum_{j=1}^{n}|x_j|\left|\partial^2_{y_j\tau}\Phi^{\ell}\big(\xi_R(t,\eta)\big)\right| +\,4\sum_{j=1}^{n}|y_j|\left|\partial^2_{x_j\tau}\Phi^{\ell}\big(\xi_R(t,\eta)\big)\right|.
\end{eqnarray*}
So,
\begin{eqnarray*}
\left|\Delta_{\mathbb{H}}\varphi_R(t,\eta)\right|&\lesssim&\Phi^{\ell-2}_*\big(\xi_R(t,\eta)\big)\left(|\Phi'_*(\xi_R(t,\eta))|^2+|\Phi''_*(\xi_R(t,\eta))|\right)\left(\left|\nabla_x\xi_R(t,\eta)\right|^2+\left|\nabla_y\xi_R(t,\eta)\right|^2\right)\\
&{}&+\,\Phi^{\ell-1}_*\big(\xi_R(t,\eta)\big)|\Phi'_*(\xi_R(t,\eta))|\left(\left|\Delta_x\xi_R(t,\eta)\right|+\left|\Delta_y\xi_R(t,\eta)\right|\right)\\
&{}&+\,(|x|^2+|y|^2)\Phi^{\ell-2}_*\big(\xi_R(t,\eta)\big)\left(|\Phi'_*(\xi_R(t,\eta))|^2+|\Phi''_*(\xi_R(t,\eta))|\right)\left|\partial_\tau\xi_R(t,\eta)\right|^2\\
&{}&+\,(|x|^2+|y|^2)\Phi^{\ell-1}_*\big(\xi_R(t,\eta)\big)|\Phi'_*(\xi_R(t,\eta))|\left|\partial^2_\tau\xi_R(t,\eta)\right|\\
&{}&+\,\Phi^{\ell-2}_*\big(\xi_R(t,\eta)\big)\left(|\Phi'_*(\xi_R(t,\eta))|^2+|\Phi''_*(\xi_R(t,\eta))|\right)\left|\partial_\tau\xi_R(t,\eta)\right|\sum_{j=1}^{n}|x_j|\left|\partial_{y_j}\xi_R(t,\eta)\right|\\
&{}&+\,\Phi^{\ell-1}_*\big(\xi_R(t,\eta)\big)|\Phi'_*(\xi_R(t,\eta))|\sum_{j=1}^{n}|x_j|\left|\partial^2_{y_j\tau}\xi_R(t,\eta)\right|\\
&{}&+\,\Phi^{\ell-2}_*\big(\xi_R(t,\eta)\big)\left(|\Phi'_*(\xi_R(t,\eta))|^2+|\Phi''_*(\xi_R(t,\eta))|\right)\left|\partial_\tau\xi_R(t,\eta)\right|\sum_{j=1}^{n}|y_j|\left|\partial_{x_j}\xi_R(t,\eta)\right|\\
&{}&+\,\Phi^{\ell-1}_*\big(\xi_R(t,\eta)\big)|\Phi'_*(\xi_R(t,\eta))|\sum_{j=1}^{n}|y_j|\left|\partial^2_{x_j\tau}\xi_R(t,\eta)\right|.
\end{eqnarray*}
Taking into account the support of $\Phi_*$, a direct computation yields
$$\left|\Delta_{\mathbb{H}}\varphi_R(t,\eta)\right|\lesssim R^{-1} \left(\varphi_R^*(t,\eta)\right)^{\frac{1}{p}}.$$
\end{proof}
%%%%%%%%%%%%%%%%%%%%%%%%%%%%%%%%%%%%%%%%%%%%%%%%%%%%%%%%%%%%%%%%%%%%%%%%%%%%%%%%%%%%%%%

%%%%%%%%%%%%%%%%%%%%%%%%%%%%%%%%%%%%%%%%%%%%%%%%%%%%%%%%%%%%%%%%%%%%%%%%%
\section{Proof of Theorem \ref{decay}}\label{sec3}
We start by proving $M_\infty>0$. From \eqref{mass}, we have
\begin{equation}\label{3.1}
    0\leq\int_{\mathbb{H}^n}u(t,\eta)\,d\eta=\int_{\mathbb{H}^n}u_0(\eta)\,d\eta - \int_0^t\int_{\mathbb{H}^n}k(s)u^p(s,\eta)\,d\eta\,ds.
\end{equation}
 Hence, for $u_0\in L^1(\mathbb{H}^n)$, we immediately obtain
\begin{equation}\label{3.2}
u\in L^\infty((0,\infty),L^1(\mathbb{H}^n))\quad\hbox{and}\quad  \int_0^\infty\int_{\mathbb{H}^n}k(t)u^p(t,\eta)\,d\eta\,dt\leq \|u_0\|_1<\infty.
\end{equation}
As $u$ is a mild solution, we immediately get the following estimate
$$
0\leq u(t)\leq S_{\mathbb{H}}(t)u_0,\qquad\hbox{for all}\,\,t> 0,
$$
which implies, using Lemma \ref{Lp-Lqestimate}, for all $t> 0$, that
\begin{eqnarray}\label{HTP}
  \|u(t)\|^p_p &\leq&\left\|S_{\mathbb{H}}(t)u_0\right\|^p_p\nonumber\\
  &\leq& \min\left\{C\,t^{-\frac{Q(p-1)}{2}}\|u_0\|^p_1,\,\|u_0\|^p_p\right\}\nonumber\\
  &\equiv &H(t,p,u_0).
\end{eqnarray}
Now, for fixed $\varepsilon\in(0,1]$, we denote by $u^\varepsilon=u^\varepsilon(x,t)$ the solution of \eqref{1} with initial condition $\varepsilon u_0$. The comparison principle implies that
$$0\leq u^\varepsilon(t,\eta)\leq u(t,\eta),\qquad\hbox{for all}\,\,\eta\in\mathbb{H}^n,\,\,t>0.$$
Hence,
$$M_\infty\geq M^\varepsilon_\infty\equiv\lim\limits_{t\rightarrow\infty}\int_{\mathbb{H}^n}u^\varepsilon(t,\eta)\,d\eta .$$
Therefore, to prove \eqref{2.4}, it suffices to show, for small $\varepsilon>0$ which will be determined later, that $M^\varepsilon_\infty >0$.
Applying \eqref{3.1} to $u^\varepsilon$ and letting $t\rightarrow\infty$, we obtain
\begin{equation}\label{3.11}
  M^\varepsilon_\infty = \varepsilon\left\{\int_{\mathbb{H}^n}u_0(\eta)\,d\eta -\frac{1}{\varepsilon} \int_0^\infty\int_{\mathbb{H}^n}k(t)\left(u^\varepsilon(t,\eta)\right)^p\,d\eta\,dt\right\}.
\end{equation}
Furthermore, by applying \eqref{HTP} to $u^\varepsilon$, we get
$$ \|u^\varepsilon(t)\|^p_p\leq H(t,p,\varepsilon u_0)=\varepsilon^pH(t,p,u_0).$$
Hence
\begin{eqnarray*}
  \frac{1}{\varepsilon}\int_0^\infty\int_{\mathbb{H}^n}k(t)\left(u^\varepsilon (t,\eta)\right)^p\,d\eta\,dt&=& \frac{1}{\varepsilon}\int_0^\infty k(t)  \|u^\varepsilon(t)\|^p_p\,dt\\
&\leq&\varepsilon^{p-1}\int_0^\infty k(t)H(t,p,u_0)\,dt.
\end{eqnarray*}
 It follows immediately from the definition of the function $H$ that
\begin{eqnarray*}
\int_{0}^{1}k(t)H(t,p,u_0)\,dt&\leq& \|u_0\|_p^p\int_0^1k(t)dt<\infty,\\
\int_{1}^{\infty}k(t)H(t,p,u_0)\,dt&\leq& C\,\|u_0\|_1^p\int_1^{\infty}t^{-\frac{Q(p-1)}{2}}k(t)dt<\infty.
\end{eqnarray*}
Consequently,
$$\int_{0}^{\infty}k(t)H(t,p,u_0)\,dt<\infty,$$
which yields to
$$
\lim_{\varepsilon\rightarrow 0^+}\frac{1}{\varepsilon} \int_0^\infty\int_{\mathbb{H}^n}k(t)\left(u^\varepsilon(t,\eta)\right)^p\,d\eta\,dt=0.
$$
Hence, as $\displaystyle \int_{\mathbb{H}^n}u_0(\eta)\,d\eta>0$, there exists $\varepsilon_0\in(0,1)$ such that
$$\frac{1}{\varepsilon_0} \int_0^\infty\int_{\mathbb{H}^n}k(t)\left(u^{\varepsilon_0}(t,\eta)\right)^p\,d\eta\,dt\leq \frac{1}{2}\int_{\mathbb{H}^n}u_0(\eta)\,d\eta,$$
which implies that
$$M_\infty\geq M^{\varepsilon_0}_\infty\geq \frac{\varepsilon_0}{2}\int_{\mathbb{H}^n}u_0(\eta)\,d\eta>0.$$
This completes the proof of \eqref{2.4}. The proof of the asymptotic behavior \eqref{2.3} will be divided into two cases.\\

\noindent \textbf{Case of $\boldsymbol{q = 1}$.}  Applying Minkowski's inequality, we decompose the expression $\|u(t)-M_\infty h_t\|_1$ as follows:
\begin{eqnarray}\label{4.1}
\|u(t)-M_\infty h_t\|_1 &\leq& \|u(t)-h_{t-t_0}\ast_{_{\mathbb{H}}}  u(t_0)\|_1+\,\|h_{t-t_0}\ast_{_{\mathbb{H}}}  u(t_0)-M_{\mathbb{H}}(t_0)h_{t-t_0}\|_1\nonumber\\
&{}&+\,\|M_{\mathbb{H}}(t_0)h_{t-t_0}-M_{\mathbb{H}}(t_0)h_t\|_1+\,\|M_{\mathbb{H}}(t_0)h_t-M_\infty h_t\|_1,
\end{eqnarray}
for all $t\geq t_0 > 0$. Since $u$ is a mild solution, using the $L^1-L^1$ estimate \eqref{Lp-Lqestimate}, we get from \eqref{IEG} that
\begin{equation}\label{4.2}
\|u(t)-h_{t-t_0}\ast_{_{\mathbb{H}}}  u(t_0)\|_1\lesssim \int_{t_0}^t k(s)\|u(s)\|_p^p ds,
\end{equation} 
for all $t\geq t_0 >0$. Furthermore, applying Lemma \ref{Taylor} with $g=u(t_0)$, we obtain
\begin{equation}\label{4.3}
\lim_{t\rightarrow\infty}\|h_{t-t_0}\ast_{_{\mathbb{H}}}  u(t_0)-M_{\mathbb{H}}(t_0)h_{t-t_0}\|_1=0.
\end{equation} 
Applying again Lemma \ref{Taylor} with $g=h_{t_0}$, and using Proposition \ref{properties} $(v)$, we find
$$
\lim_{t\rightarrow\infty}\|h_{t-t_0}-h_{t-t_0}\ast_{_{\mathbb{H}}}  h_{t_0}\|_1=0,
$$ 
which implies, using Proposition \ref{properties} $(ix)$, that
\begin{equation}\label{4.4}
\|M_{\mathbb{H}}(t_0)h_{t-t_0}-M_{\mathbb{H}}(t_0)h_t\|_1\leq \|u_0\|_1\|h_{t-t_0}-h_t\|_1=\|u_0\|_1\|h_{t-t_0}-h_{t-t_0}\ast_{_{\mathbb{H}}}  h_{t_0}\|_1\longrightarrow 0,
\end{equation}
when $t$ goes to $\infty$. In addition, from Proposition \ref{properties} $(v)$, we directly have
\begin{equation}\label{4.5}
\|M_{\mathbb{H}}(t_0)h_t-M_\infty h_t\|_1\leq |M_{\mathbb{H}}(t_0)-M_\infty|.
\end{equation} 

Substituting estimates \eqref{4.2}-\eqref{4.5} into \eqref{4.1}, we deduce
\begin{equation}
\limsup_{t\rightarrow\infty}\|u(t)-M_\infty h_t\|_1\lesssim\int_{t_0}^\infty k(s)\|u(s)\|_p^p ds+|M_{\mathbb{H}}(t_0)-M_\infty|.
\end{equation}
Letting $t_0\to \infty$ and using \eqref{3.2}, it follows that 
\begin{equation}
\lim\limits_{t\to\infty}\|u(t)-M_\infty h_t\|_1=0.
\end{equation}
\noindent \textbf{Case of $\boldsymbol{q > 1}$.}  On the one hand, for each $m\in[1,+\infty]$, using the $L^1-L^m$ estimate \eqref{Lp-Lqestimate}, we arrive at
\begin{equation}\label{E1}
\|u(t)\|_m\leq\left\|S_{\mathbb{H}}(t)u_0\right\|_m \leq  Ct^{-\frac{Q}{2}\left(1-\frac{1}{m}\right)}\|u_0\|_1.
\end{equation}
In addition, by Proposition \ref{pointwiseupperbound} (or Proposition \ref{pointwiseheatestimation}) along with Proposition \ref{properties} $(v)$ , we obtain the decay estimate 
\begin{eqnarray}\label{E2}
\|h_t\|_m\lesssim \,t^{-\frac{Q(m-1)}{2 m}}.
\end{eqnarray}
%\begin{eqnarray}\label{E2}
%\|h_t\|_m\leq \|h_t\|^{\frac{m-1}{m}}_\infty\|h_t\|_1\lesssim \,t^{-\frac{Q(m-1)}{2 m}}\|h_t\|_1=t^{-\frac{Q(m-1)}{2 m}},
%\end{eqnarray}
On the other hand, for a fixed $m\in[1,+\infty]$ and for every $q\in [1,m)$, by employing Minkowski's inequality followed by H\"older's inequality, one obtains
\begin{equation}\label{final}
\|u(t)-M_\infty h_t\|_q\leq \|u(t)-M_\infty h_t\|_1^{1-\delta}\left(\|u(t)\|_m^\delta+\|M_\infty h_t\|_m^\delta\right),
\end{equation}  
where $\delta=(1-1/q)/(1-1/m)$. Inserting \eqref{E1} and \eqref{E2} into \eqref{final} yields
$$
t^{\frac{Q}{2}\left(1-\frac{1}{q}\right)}\|u(t)-M_\infty
h_t\|_q\lesssim \|u(t)-M_\infty h_t\|_1^{1-\delta}\longrightarrow 0,\quad\hbox{as}\;\; t\to\infty.
$$
This completes the proof of Theorem \ref{decay}.\hfill$\square$

%%%%%%%%%%%%%%%%%%%%%%%%%%%%%%%%%%%%%%%%%%%%%%%%%%%%%%%%%%%%%%%%%%%%%%%%%%%%%%%%%%%%%%%

\section{Proof of Theorem \ref{convto0}}\label{sec4}
We assume $M_{\infty}>0$ and argue by contradiction. The idea of the proof is to use the variational formulation of the weak solution by choosing the appropriate test function. As $u$ is a global mild solution, then $u$ is a global weak solution and therefore
 \begin{eqnarray}\label{newweaksolution}
&{}&\int_{\mathbb{H}^n}u(T,\eta)\varphi(T,\eta)\,d\eta-\int_{\mathbb{H}^n}u_0(\eta)\varphi(0,\eta)\,d\eta+\int_0^T\int_{\mathbb{H}^n}k(t)u^p\varphi\,d\eta\,dt\\
&{}&=\int_0^T\int_{\mathbb{H}^n}u\left[\partial_t\varphi+\Delta_{\mathbb{H}}\varphi \right]\,d\eta\,dt\nonumber,
\end{eqnarray}
 for all $T>0$ and any test function $\varphi\in W^{1,\infty}([0,T),W^{2,\infty}(\mathbb{H}^n))$. Let $\varphi_R$ be a test function as defined in \eqref{varphi_R}. Consequently, for every $T>R$, 
\begin{align*}
&
M_{\mathbb{H}}(u(T))-\int_{\mathbb{H}^n}u_0(\eta)\varphi_R(0,\eta)\,d\eta
+\int_0^T\int_{\mathbb{H}^n}k(t)(u(t,\eta))^p\varphi_R(t,\eta)\,d\eta\,dt
\\&=\int_0^T\int_{\mathbb{H}^n}u(t,\eta)\Big(\partial_t\varphi_R(t,\eta)+\Delta_{\mathbb{H}}\varphi_R(t,\eta)\Big)\,d\eta\,dt
\\
&\leq 
\frac{C}{R}
\int_0^T\int_{\mathbb{H}^n}u(t,\eta)\varphi_R^*(t,\eta)^{\frac{1}{p}}
\,d\eta\,dt,
\end{align*}
where we have used $\varphi_R(T,\eta)=1$ and Lemma \eqref{lem:testfunction}. Letting $T\rightarrow\infty$, 
using the dominated convergence theorem, we arrive at
\begin{equation}\label{newweaksolution1}
M_\infty-\int_{\mathbb{H}^n}u_0(\eta)\varphi_R(0,\eta)\,d\eta+\int_0^\infty\int_{\mathbb{H}^n}k(t)(u(t,\eta))^p\varphi_R(t,\eta)\,d\eta\,dt\leq  \frac{C}{R} \int_0^\infty\int_{\mathbb{H}^n}u(t,\eta)\varphi_R^*(t,\eta)^{\frac{1}{p}}
\,d\eta\,dt.
\end{equation}
Since $\varphi_R(t,\eta)\rightarrow 0$ when $R\rightarrow\infty$, we conclude that
$$-\int_{\mathbb{H}^n}u_0(\eta)\varphi_R(0,\eta)\,d\eta+\int_0^\infty\int_{\mathbb{H}^n}k(t)(u(t,\eta))^p\varphi_R(t,\eta)\,d\eta\,dt\longrightarrow 0,\qquad\hbox{when}\,\,R\rightarrow\infty,$$
where we have used the dominated convergence theorem together with \eqref{3.2}, therefore there exists $R_1>0$ such that
\begin{equation}\label{newweaksolution2}
M_\infty-\int_{\mathbb{H}^n}u_0(\eta)\varphi_R(0,\eta)\,d\eta+\int_0^\infty\int_{\mathbb{H}^n}k(t)u^p\varphi_R(t,\eta)\,d\eta\,dt\geq\frac{M_\infty}{2}>0,\quad\hbox{for all}\,\,R\geq R_1.
\end{equation}
Combining \eqref{newweaksolution2} and \eqref{newweaksolution1}, we get
\begin{equation}\label{newweaksolution3}
\frac{M_\infty}{2}\leq  \frac{C}{R} \int_0^\infty\int_{\mathbb{H}^n}u(t,\eta)\varphi_R^*(t,\eta)^{\frac{1}{p}}\,d\eta\,dt,\quad\hbox{for all}\,\,R\geq R_1.
\end{equation}
By H\"older's inequality, we arrive at 
\begin{eqnarray}\label{testfuntionmethod}
 \frac{M_{\infty}}{2} &\leq&  \frac{C}{R} \left(\int_0^\infty\int_{\mathbb{H}^n}\mathbbm{1}_{[\frac{1}{2},1]}(\xi_R(t,\eta))\,d\eta\,dt
\right)^{1-\frac{1}{p}}\left(\int_0^\infty\int_{\mathbb{H}^n}(u(t,\eta))^p\varphi_R^*(t,\eta)\,d\eta\,dt\right)^{\frac{1}{p}}\nonumber\\
&\lesssim&  R^{\frac{Q}{2p}(p-1)-\frac{1}{p}}\left(\int_0^\infty\int_{\mathbb{H}^n}(u(t,\eta))^p\varphi_R^*(t,\eta)\,d\eta\,dt\right)^{\frac{1}{p}},
\end{eqnarray}
for every $R\geq R_1$. To proceed, we introduce the auxiliary function
\[
Y(R)=
\int_R^\infty \left(\int_0^\infty\int_{\mathbb{H}^n}(u(t,\eta))^p\varphi_\rho^*(t,\eta)\,d\eta\,dt\right)\frac{d\rho}{\rho}, \quad R>0,
\]
which is well-defined and bounded.  Indeed, we have
$$Y(R)\leq \int_0^\infty\int_0^\infty\int_{\mathbb{H}^n}(u(t,\eta))^p \mathbbm{1}_{[\frac{1}{2},1]}\big(\xi_\rho(t,\eta)\big)\,d\eta\,dt\frac{d\rho}{\rho}=\int_0^\infty\int_{\mathbb{H}^n}(u(t,\eta))^p\int_0^\infty \mathbbm{1}_{[\frac{1}{2},1]}\big(\xi_\rho(t,\eta)\big)\frac{d\rho}{\rho}\,d\eta\,dt.$$
Observe that
$$\int_0^\infty\mathbbm{1}_{[\frac{1}{2},1]}\big(\xi_\rho(t,\eta)\big)\frac{d\rho}{\rho}=\int_0^\infty\mathbbm{1}_{[\frac{1}{2},1]}(s)\frac{ds}{s}=\int_{1/2}^1\frac{ds}{s}=\log 2,$$
and therefore we conclude that
$$Y(R)\leq \log 2\int_0^\infty\int_{\mathbb{H}^n}(u(t,\eta))^p\,d\eta\,dt\leq C,\quad\hbox{for all}\,\,R> 0,$$
where we have used estimate \eqref{3.2} and the fact that $\displaystyle \inf_{t\geq 0}k(t)>0$. On the other hand,
$$Y'(R)=-R^{-1}\int_0^\infty\int_{\mathbb{H}^n}(u(t,\eta))^p\varphi_R^*(t,\eta)\,d\eta\,dt,\qquad\hbox{for all}\,\,R> 0,$$
which yields, using \eqref{testfuntionmethod}, that
$$(M_\infty)^p\lesssim -R^{\frac{Q}{2}(p-1)} Y'(R),\qquad\hbox{for all}\,\,R\geq R_1,$$
and then
$$Y'(R) \lesssim -(M_\infty)^p\,R^{-\frac{Q}{2}(p-1)},\qquad\hbox{for all}\,\,R\geq R_1.$$
Integrating on $[\rho,R]$, for any $\rho\geq R_1$, we get
$$Y(R)-Y(\rho) \lesssim -(M_\infty)^p\int_{\rho}^{R}s^{-\frac{Q}{2}(p-1)}\,ds,\qquad\hbox{for all}\,\,R\geq \rho\geq R_1,$$
which implies
$$0\leq Y(R)\lesssim Y(\rho) -(M_\infty)^p\int_{\rho}^{R}s^{-\frac{Q}{2}(p-1)}\,ds,\qquad\hbox{for all}\,\,R\geq \rho\geq R_1.$$
Hence
\begin{equation}\label{4.4}
(M_\infty)^p\int_{\rho}^{R}s^{-\frac{Q}{2}(p-1)}\,ds \lesssim Y(\rho),\qquad\hbox{for all}\,\,R\geq \rho\geq R_1.
\end{equation}
Using the boundedness of $Y(R_1)$ and \eqref{4.4}, we conclude that
$$\int_{R_1}^{R}s^{-\frac{Q}{2}(p-1)}\,ds \lesssim(M_\infty)^{-p} Y(R_1)\lesssim 1,\qquad\hbox{for all}\,\,R\geq R_1.$$
and therefore
$$\int_{R_1}^{\infty}s^{-\frac{Q}{2}(p-1)}\,ds \lesssim 1.$$
However, this leads to a contradiction since the integral on the left-hand side diverges when $p\leq 1+\frac{2}{Q}$. This completes the proof of Theorem \ref{convto0}. \hfill$\square$

%%%%%%%%%%%%%%%%%%%%%%%%%%%%%%%%%%%%%%%%%%%%%%%%%%%%%%%%%%%%%%%%%%%%%%%%%%%%%%%%%%%%%%%

\end{document}